\documentclass{amsart}
\usepackage{graphicx}
\usepackage{latexsym}
\usepackage{hyperref}
\usepackage{amsfonts}
\usepackage[all]{xy}
\usepackage{amssymb, mathrsfs, amsfonts, amsmath}
\usepackage{amsbsy}
\usepackage{amsfonts}
\newtheorem*{acknowledgement}{Acknowledgement}
\newtheorem{corollary}{Corollary}

\newtheorem{lemma}{Lemma}
\newtheorem{proposition}{Proposition}

\newtheorem{theorem}{Theorem}

\numberwithin{equation}{section}

\begin{document}

\title[Einstein warped products]{Bach-flat noncompact steady\\ quasi-Einstein manifolds}
\author{M. Ranieri}
\author{E. Ribeiro Jr} 

\address[M. Ranieri]{Universidade Federal de Alagoas - UFAL, Instituto de Matem\'atica, CEP 57072-970-Macei\'o / AL, Brazil} \email{ranieri2011@gmail.com}

\address[E. Ribeiro]{Universidade Federal do Cear\'a - UFC, Departamento  de Matem\'atica, Campus do Pici, Av. Humberto Monte, Bloco 914, CEP 60455-760-Fortaleza / CE, Brazil}\email{ernani@mat.ufc.br}

\thanks{E. Ribeiro was partially supported by CNPq/Brazil}

\thanks{M. Ranieri was partially supported by CAPES/Brazil}

\date{September 21, 2016}

\newcommand{\spacing}[1]{\renewcommand{\baselinestretch}{#1}\large\normalsize}
\spacing{1.2}

\begin{abstract}
The goal of this article is to study the geometry of Bach-flat noncompact steady quasi-Einstein manifolds. We show that a Bach-flat noncompact steady quasi-Einstein manifold $(M^{n},\,g)$ with positive Ricci curvature such that its potential function has at least one critical point must be a warped product with Einstein fiber. In addition, the fiber has constant curvature if $n = 4.$
\end{abstract}

\maketitle

\section{Introduction}

In 1958, Ren\'e Thom posed the following well-known question: ``{\it Are there any best Riemannian structures on a smooth manifold?}". The best Riemannian structures on a given manifold are those of constant curvature. In this spirit, a Riemannian manifold of dimension greater than 2 with constant Ricci curvature is called {\it Einstein.} Hilbert and Einstein proved that the critical metrics of the total scalar curvature functional restricted to the set of smooth Riemannian structures on a compact manifold of unitary volume are Einstein. We remark that Einstein manifolds are not only fascinating in themselves but are also related to many important topics of Riemannian geometry. For a comprehensive reference on such a subject, we refer the reader to \cite{Besse}.

A classical problem in Riemannian geometry is to construct new explicit examples of Einstein metrics. According to ``Besse's book" \cite{Besse}, a promising way for that purpose is that of warped products. The $m$-Bakry-Emery Ricci tensor, which appeared
previously in \cite{bakry} and \cite{Qian}, is useful as an attempt to better understand Einstein warped product. More precisely, the $m$-Bakry-Emery Ricci tensor is given by
\begin{equation}
\label{bertens}
Ric_{f}^{m}=Ric+\nabla ^2f-\frac{1}{m}df\otimes df,
\end{equation} where $f$ is a smooth function on $M^n$ and $\nabla ^2f$ stands for the Hessian of $f.$ We highlight that it is also used to study the weighted measure $d\mu=e^{-f}dx,$ where $dx$ is the Riemann-Lebesgue measure determined by the metric. 

A complete Riemannian manifold $(M^n,\,g),$ $n\geq 2,$ will be called $m$-{\it quasi-Einstein manifold}, or simply {\it quasi-Einstein manifold}, if there exist a smooth potential function $f$
on $M^n$ and a constant $\lambda$ satisfying the following fundamental equation
\begin{equation}
\label{eqqem}
Ric_{f}^{m}=Ric+\nabla ^2f-\frac{1}{m}df\otimes
df=\lambda g,
\end{equation} where $\nabla ^2 f$ stands for the Hessian of $f.$ 

It is easy to see that a $\infty$-quasi-Einstein manifold means a gradient Ricci soliton. Ricci solitons model the formation of singularities in the Ricci flow and  correspond to self-similar solutions, i.e., solutions which evolve along symmetries of the flow, see \cite{Cao} and references therein for more details on this subject. On the other hand, when $m$ is a positive integer it
corresponds to a warped product Einstein metric, see,
for instance, \cite{CaseShuWey,Kim}. We also remark that $1$-quasi-Einstein manifolds are more commonly called {\it static metrics} and such metrics have connections to scalar curvature, the positive mass theorem and general relativity. Recall that a quasi-Einstein metric $g$ on a manifold $M^n$ will be
called \emph{expanding}, \emph{steady} or \emph{shrinking},
respectively, if  $\lambda<0,\,\lambda=0$ or $\lambda>0$. Moreover,
a quasi-Einstein manifold will be called \emph{trivial} if its potential function $f$ is constant, otherwise it will be \emph{nontrivial}. Notice that
the triviality implies that $M^n$ is an Einstein manifold.

According to \cite{Ernani2,Besse,CaseShuWey,Kim} and \cite{rimoldi} the remarkable motivation to study quasi-Einstein metrics on a Riemannian manifold is its direct relation with the existence of Einstein warped product, which also have different properties compared with the gradient Ricci solitons. In this sense, it is important to recall that, on a quasi-Einstein manifold, there is an indispensable constant  $\mu$ such that 
\begin{equation}\label{2eq}
\Delta f-|\nabla f|^{2} = m\lambda-m\mu e^{\frac{2}{m}f}.
\end{equation} See \cite{Ernani2,Kim,rimoldi} and \cite{Wang1} for a comprehensive treatment of this feature. 

Qian \cite{Qian} proved that shrinking quasi-Einstein manifolds must be compact. Moreover, from Kim-Kim \cite{Kim} the converse statement remains true. Thereby,  it is now well-known that a quasi-Einstein manifold is compact if and only if $\lambda>0.$ An example of nontrivial quasi-Einstein manifold with $\lambda>0,$ $m>1$ and $\mu>0$ was obtained in \cite{LuePage}. Some examples of expanding quasi-Einstein manifolds with arbitrary $\mu$ as well as steady quasi-Einstein manifolds with $\mu>0$ were constructed in \cite{BRS,Besse} and \cite{Wang1}. At the same time, Case \cite{Case} has shown that steady quasi-Einstein manifolds with $\mu\leq0$ are trivial. See also \cite{rimoldi2} for further results related.

In order to proceed it is important to remember that the Bach tensor on a Riemannian manifold $(M^n,g),$ $n\geq 4,$ which was introduced to study conformal relativity in \cite{bach}, is defined in terms of the components of the Weyl tensor $W_{ikjl}$ as follows
\begin{equation}
\label{bach} B_{ij}=\frac{1}{n-3}\nabla^{k}\nabla^{l}W_{ikjl}+\frac{1}{n-2}R_{kl}W_{i}\,^{k}\,_{j}\,^{l},
\end{equation} while for $n=3$ it is given by
\begin{equation}
\label{bach3} B_{ij}=\nabla^kC_{kij},
\end{equation} where $C_{ijk}$ stands for the Cotton tensor. We say that $(M^n,g)$ is Bach-flat when $B_{ij}=0.$ It is straightforward to check that locally conformally flat metrics as well as Einstein metrics are Bach-flat. In addition, for dimension $n=4,$ it is well-known that half-conformally flat or locally conformally to an Einstein manifold implies Bach-flat. However, Leistner and Nurowski \cite{LN} obtained a large class of Bach-flat examples which are not conformally Einstein; for more details we address to \cite{Besse}.

Recently, Cao and Chen  \cite{caoshinking} have studied Bach-flat gradient Ricci solitons.  They obtained a stronger classification for shrinking gradient Ricci solitons under the Bach-flat assumption. Afterward, Cao, Catino, Chen, Mantegazza and Mazzieri \cite{caosteady} were able to show that any $n$-dimensional $(n \ge 4)$ complete Bach-flat gradient steady Ricci soliton with positive Ricci curvature such that the scalar curvature $R$ attains its maximum at some interior point must be isometric to the Bryant soliton. For more details, we refer the reader \cite{caosteady,caoshinking} and \cite{caoW0}. The Bach-flat assumption was also studied in another special metrics, see, for instance \cite{ernani,ying,CMMR} and \cite{jiewei}.

In light of the previous results, it is natural to ask what occurs on quasi-Einstein manifolds. As it was previously mentioned a quasi-Einstein manifold is compact if and only if $\lambda>0.$ In that case, Chen and He \cite{QiangHe} proved that a Bach-flat shrinking quasi-Einstein manifold is either Einstein or a finite quotient of a warped product with $(n-1)$-dimensional Einstein fiber. In this paper, mainly inspired by \cite{caoW0} as well as \cite{caosteady}, we shall focus our attention on Bach-flat steady quasi-Einstein manifolds. In particular, the manifold must be noncompact. More precisely, we shall provide a classification result for Bach-flat noncompact steady quasi-Einstein manifolds with positive Ricci curvature. A crucial ingredient here that should be emphasized is a pinching estimate for the function  $u = e^{-\frac{f}{m}}$ (cf. Lemma \ref{lemmaforu}).

After these preliminary remarks we may announce our first result as follows.

\begin{theorem}\label{theoremcottonvanishes}
Let $(M^n,\,g,\,f,\,m>1),$ $n\geq 4,$ be a Bach-flat noncompact steady quasi-Einstein manifold with positive Ricci curvature such that $f$ has at least one critical point. Then $M^n$ has harmonic Weyl tensor and $W_{ijkl}\nabla^l f = 0.$
\end{theorem}

At the same time, it is worth to point out that $4$-dimensional manifolds have special behavior; see \cite{Besse} for more information about this specific dimension. In such a dimension we have established the following result.

\begin{theorem}\label{theoremweylvanishes}
Let $(M^4,\,g,\,f,\,m>1)$ be a $4$-dimensional Bach-flat noncompact steady quasi-Einstein manifold with positive Ricci curvature such that $f$ has at least one critical point. Then $M^4$ is locally conformally flat.
\end{theorem}

Next, as an application of Theorems \ref{theoremcottonvanishes} and \ref{theoremweylvanishes}, jointly with Theorem 1.2 in \cite{Petersen}, we have the following classification result.

\begin{corollary}
Let $(M^n,\,g,\,f,\,m>1),$ $n\geq 4,$ be a Bach-flat noncompact steady quasi-Einstein manifold with positive Ricci curvature such that $f$ has at least one critical point. Then $(M^n,\,g)$ is a warped product with
$$g = dt^2 + \psi^2(t)g_{_{L}}\,\,\,\,\hbox{and}\,\, \,\,f = f(t),$$ where $g_{_{L}}$ is Einstein of non-negative Ricci curvature. In addition, the fiber has constant curvature if $n = 4.$
\end{corollary}

\section{Background and Key Lemmas}

In this section we shall present a couple of  lemmas that will be useful in the proof of our main results.  We begin recalling that the Weyl curvature $W_{ijkl}$ is defined by the following decomposition formula
\begin{eqnarray}
\label{riemanntensor}
R_{ijkl}&=&W_{ijkl}+\frac{1}{n-2}\big(R_{ik}g_{jl}+R_{jl}g_{ik}-R_{il}g_{jk}-R_{jk}g_{il}\big) \nonumber\\
 &&-\frac{R}{(n-1)(n-2)}\big(g_{jl}g_{ik}-g_{il}g_{jk}\big),
\end{eqnarray} where $R_{ijkl}$ stands for the Riemann curvature tensor. Moreover, the Cotton tensor $C_{ijk}$ is given by
\begin{equation}
\label{cottontensor} \displaystyle{C_{ijk}=\nabla_{i}R_{jk}-\nabla_{j}R_{ik}-\frac{1}{2(n-1)}\big(\nabla_{i}R
g_{jk}-\nabla_{j}R g_{ik}).}
\end{equation} It is easy to check that $C_{ijk}$ is skew-symmetric in the first two indices and trace-free in any two indices. We also remember that $W_{ijkl}$ and $C_{ijk}$ are related as follows
\begin{equation}\label{relationcottonandweyl}
-\frac{(n-3)}{(n-2)}C_{ijk} = \nabla^l W_{ijkl}.
\end{equation} Moreover, taking into account (\ref{relationcottonandweyl}) we may extend the definition of the Bach tensor for $n\geq 3$  by
\begin{equation}\label{bachwithcotton}
B_{ij} = \frac{1}{n-2}\big(\nabla^kC_{kij} + R_{kl}W_{ikjl}\big).
\end{equation} Since $W \equiv 0$ in dimension three, for $n=3$ we have
\begin{equation*}
B_{ij}=\nabla^kC_{kij}.
\end{equation*}

Following the notation employed in \cite{QiangHe}, in the spirit of  \cite{caoshinking}, we recall that the covariant $3$-tensor $D$ is given by

\begin{eqnarray}\label{definitiond}
D_{ijk} & = & \frac{1}{n-2}(R_{jk}\nabla_i f - R_{ik}\nabla_j f) + \frac{1}{(n-1)(n-2)}(R_{il}\nabla^l f g_{jk} - R_{jl}\nabla^l f g_{ik}) \nonumber \\
& & -\frac{R}{(n-1)(n-2)}(g_{jk}\nabla_i f - g_{ik}\nabla_j f).
\end{eqnarray} It is not difficult to check that the tensor $D_{ijk}$ is skew-symmetric in their first two indices and trace-free in any two indices:
\begin{equation}\label{symmetrytensord}
D_{ijk} = -D_{jik}\,\,\,\,\,\hbox{and}\,\,\,\,\, g^{ij}D_{ijk} = g^{ik}D_{ijk} = 0.
\end{equation}

In order to set the stage for the proof to follow let us recall some useful results obtained in  \cite{QiangHe}. Indeed, taking into account (\ref{definitiond}) we shall show a relation between the Cotton tensor and the Weyl tensor on a quasi-Einstein manifold.

\begin{lemma}[Chen-He \cite{QiangHe}] Let $(M^{n},\,g,\,f)$ be a quasi-Einstein manifold. Then we have:
\begin{equation}\label{cottonwithd}
C_{ijk} = \frac{m+n-2}{m}D_{ijk} - W_{ijkl}\nabla^l f.
\end{equation}
\end{lemma}

We also need of the following results by Chen-He \cite{QiangHe}.

\begin{lemma}[Chen-He \cite{QiangHe}]
Let $(M^{n},\,g,\,f)$ be a quasi-Einstein manifold. Assume that $\Sigma$ is a level set of $f$ with $\nabla f(p) \neq 0.$ Then we have:
\begin{equation}
|D|^2 = \frac{2|\nabla f|^4}{(n-2)^2}\sum_{a,b=2}^n |h_{ab}-\frac{H}{n-1}g_{ab}|^2 + \frac{m^2}{2(n-1)(n-2)(m-1)^2}|\nabla^{\Sigma}R|^2,
\end{equation} where $h_{ab}$ stands for the second fundamental form of $\Sigma$ and $H$ is its mean curvature.
\end{lemma}

The next result shows that the vanishing of the tensor $D_{ijk}$ implies interesting rigidity properties about the geometry of the level surfaces of the potential function.

\begin{proposition}[Chen-He \cite{QiangHe}]\label{levelcurvesgeometry}
Let $(M^{n},\,g,\,f,\,m>1)$ be a quasi-Einstein manifold with $D_{ijk}=0.$ Let $c$ be a regular value of $f$ and $\Sigma = \{ p \in M \ | \ f(p) = c\}$ be a level hypersurface of $f.$ We also consider $e_1 = \frac{\nabla f}{|\nabla f|}$ and choose an orthonormal frame $\{e_2,...,e_n\}$ tangent to $\Sigma$. Then: 
\begin{enumerate}
\item the scalar curvature $R$ and $|\nabla f|^2$ of $(M^{n},\,g,\,f)$ are constant on $\Sigma;$
\item $R_{1a} = 0$ for $\geq 2$ and $e_1$ is an eigenvector of $Ric;$
\item on $\Sigma$, the Ricci tensor either has a unique eigenvalue or, two distinct eigenvalues with multiplicity $1$ and $n-1,$ moreover the eigenvalue with multiplicity $1$ is in the direction of $\nabla f;$
\item the second form fundamental $h_{ab}$ of $\Sigma$ is $h_{ab} = \frac{H}{n-1}g_{ab}$;
\item the mean curvature $H$ is constant on $\Sigma;$
\item $R_{1abc} = 0$, for $a,b,c \in \{2, ..., n\}.$
\end{enumerate} 
\end{proposition}

Before preceeding, it is important to remember some classical equations concerning quasi-Einstein manifolds. First of all, considering the function $u = e^{-\frac{f}{m}}$ on $M^n$, we immediately get $$\nabla u = -\frac{u}{m}\nabla f$$ as well as

\begin{equation}
Hess f - \frac{1}{m}df \otimes df = -\frac{m}{u}Hess u.
\end{equation} Taking into account (\ref{eqqem}) and (\ref{2eq}), it is easy to obtain
\begin{equation}\label{fundamentalequationforu}
\frac{u^2}{m}(R-\lambda n) + (m-1)|\nabla u|^2 = -\lambda u^2 + \mu
\end{equation} We also remember that, by Wang \cite{Wang1}, if $\lambda \leq 0$, then $R \geq \lambda n.$ Now we turn our attention for steady case. Whence, it follows from (\ref{fundamentalequationforu}) that

\begin{equation}\label{nablau}
|\nabla u|^2 \leq \frac{\mu}{m-1}
\end{equation} and also
\begin{equation}\label{scalarlimitationwithu}
u^2R \leq m\mu.
\end{equation}

In the sequel we investigate the asymptotic behavior of the function $u = e^{-\frac{f}{m}}.$ More precisely, we prove a pinching estimate which plays a central role in this work.

\begin{lemma}\label{lemmaforu}
Let $(M^n,\,g,\,f,\,m>1)$ be a complete noncompact steady quasi-Einstein manifold with positive Ricci curvature such that $f$ has at least one critical point. Then, there exist positive constants $c_1$ and $c_2$ such that the function $u = e^{-\frac{f}{m}}$ satisfies the following estimates
\begin{equation}\label{estimativateorema}
c_1r(x) - c_2 \leq u(x) \leq \sqrt{\frac{\mu}{m-1}}r(x) + |u(p)|,
\end{equation} where $p$ is a critical point of $f$ and $r(x)$ is the distance function from $p.$
\end{lemma}

\begin{proof} The proof will follow \cite{caoW0} (cf. Proposition 2.3 in \cite{caoW0}). Firstly, notice that by (\ref{nablau}) the upper bound in (\ref{estimativateorema}) in fact occurs for noncompact steady quasi-Einstein manifolds in general. 

Now, we deal of the lower bound. To do so, we first notice that (\ref{eqqem}) and (\ref{fundamentalequationforu}) yields
\begin{equation}Ric = \frac{m}{u}Hess u. \label{quasiforu}
\end{equation} We now assume that $p$ is a critical point of $f.$ So, taking into account that $M^n$ has positive Ricci curvature and $u>0,$ we immediately deduce from (\ref{quasiforu}) that $u$ is a strictly convex function. We then consider any minimizing normal geodesic $\gamma(s),$ $0 \leq s \leq s_0,$ for sufficiently large $s_0 > 0,$ starting from the point $p= \gamma(0).$ Further, denote by $X(s) = \dot{\gamma}(s)$ the unit tangent vector along $\gamma$ and $\frac{Du}{dt} = \dot{u} = \nabla_X u(\gamma(s)).$ With these notations in mind, we may use (\ref{quasiforu}) to achieve

\begin{equation}\label{eqhessu}
\nabla_X \dot{u} = \nabla_X \nabla_X u = \frac{u}{m} Ric(X,X).
\end{equation} Remembering that $\nabla u = -\frac{u}{m}\nabla f,$ it follows that a critical point of $f$ is also critical point of $u.$ Therefore, upon integrating (\ref{eqhessu}) along $\gamma,$ for $s \geq 1,$ we arrive at
\begin{equation}\label{ulinegamma}
\dot{u}(\gamma(s)) = \int_0^s  \frac{u}{m} Ric(X,X)ds \geq \int_0^1  \frac{u}{m} Ric(X,X)ds \geq c_1,  
\end{equation} where $$c_1 = \frac{c}{m}\min_{B_p(1)}u(x)$$ and $c>0$ is the least eigenvalue of Ricci curvature on the unit geodesic ball $B_p(1).$ 

Proceeding, on integrating (\ref{ulinegamma}) from $1$ to $s_0$ we get
\begin{eqnarray*}
u(\gamma(s_0)) &=& \int_1^{s_0}\dot{u}(s)ds + u(\gamma(1)) \nonumber\\&\geq& c_1s_0 - c_1 + u(\gamma(1))\nonumber\\&\geq& c_1s_0 - c_2.
\end{eqnarray*} This is what we wanted to prove.
\end{proof}

As an immediate application of Lemma \ref{lemmaforu} we have the following result.

\begin{corollary}
Let $(M^n,\,g,\,f,\,m>1)$ be a complete noncompact steady quasi-Einstein manifold with positive Ricci curvature such that $f$ has at least one critical point. Then $M^n$ is diffeomorphic to $\mathbb{R}^n$.
\end{corollary}

\begin{proof} 
From Eq. (\ref{estimativateorema}) we immediately have that $u$ is a proper function. In particular, Eq. (\ref{quasiforu}) implies that $u$ is strictly convex and then it is well-known that $M^n$ is diffeomorphic to $\mathbb{R}^n.$  
\end{proof}

Now, we use Lemma \ref{lemmaforu} to prove the main result of this section. It plays a crucial role in the proof of Theorem \ref{theoremcottonvanishes}.

\begin{lemma}\label{bachimpliesd}
Let $(M^n,\,g,\,f,\,m>1),$ $n\geq 4,$ be a Bach-flat noncompact steady quasi-Einstein manifold with positive Ricci curvature such that $f$ has at least one critical point. Then the tensor $D$ vanishes identically.
\end{lemma}

\begin{proof} To start with, we combine (\ref{bachwithcotton}) and (\ref{cottonwithd}) to obtain
\begin{eqnarray*} 
(n-2)B_{ij} & = & \nabla^kC_{kij} + W_{ikjl}R_{kl} \nonumber \\
& =& \nabla^k\left( \frac{m+n-2}{m}D_{kij} - W_{kijl}\nabla^l f\right) + W_{ikjl}R_{kl} \nonumber \\
&= & \frac{m+n-2}{m}\nabla^kD_{kij} - (\nabla^k W_{kijl})\nabla^l f - W_{kijl}\nabla^k\nabla^l f + W_{ikjl}R^{kl}. \nonumber 
\end{eqnarray*} Then, using (\ref{eqqem}) and (\ref{relationcottonandweyl}), we arrive at

\begin{eqnarray} \label{generalrelationbach}
(n-2)B_{ij} & = & \frac{m+n-2}{m}\nabla^kD_{kij} + \frac{n-3}{n-2}C_{lji}\nabla^l f\nonumber\\&& - \frac{1}{m}W_{kijl}\nabla^k f \nabla ^l f.
\end{eqnarray} Recall that $\nabla u = -\frac{1}{m}u\nabla f$ and this substituted into (\ref{generalrelationbach}) yields
\begin{equation}\label{bachwithu}
(n-2)B_{ij}\nabla^i u\nabla^ju e^{-u}u^3 = \frac{m+n-2}{m}(\nabla^k D_{kij})\nabla^i u\nabla^j ue^{-u}u^3. 
\end{equation}

On the other hand, a straightforward computation gives
\begin{eqnarray*}
\nabla^k(D_{kij}\nabla^i u \nabla^j u e^{-u}u^3) & = & (\nabla^k D_{kij})\nabla^i u \nabla^j u e^{-u}u^3 + D_{kij}(\nabla^k\nabla^i u) \nabla^j u e^{-u}u^3 \\
& &+  D_{kij}\nabla^i u(\nabla^k \nabla^j u) e^{-u}u^3\\
& = & (\nabla^k D_{kij})\nabla^i u \nabla^j u e^{-u}u^3 + D_{kij}\left(\frac{u}{m}R^{ki}\right)\nabla^j ue^{-u}u^3\\
& & + D_{kij} \left(\frac{u}{m}R^{kj}\right) \nabla^i ue^{-u}u^3,
\end{eqnarray*} where we have used Eq. (\ref{quasiforu}) in the last step. Therefore, returning to Eq. (\ref{bachwithu}) we immediately achieve
\begin{eqnarray}\label{bachwithuandd}
(n-2)B_{ij}\nabla^i u\nabla^ju e^{-u}u^3 & = & \frac{m+n-2}{m}\nabla^k(D_{kij}\nabla^i u\nabla^j u e^{-u}u^3)\nonumber\\
& & -\frac{m+n-2}{m^2}D_{kij}e^{-u}u^4(R^{ki}\nabla^j u  +R^{kj}\nabla^i u). 
\end{eqnarray}

Notice that $\nabla u = -\frac{u}{m}\nabla f$ substituted into (\ref{definitiond}) provides

\begin{eqnarray}\label{tensordwithu}
-\frac{u}{m}D_{ijk} & = & \frac{1}{n-2}(R_{jk}\nabla_i u - R_{ik}\nabla_j u)\nonumber\\ && +  \frac{1}{(n-1)(n-2)}\Big[R_{il}\nabla^l u g_{jk}- R_{jl}\nabla^l u g_{ik}\nonumber\\&&- R(g_{jk}\nabla_i u - g_{ik}\nabla_j u)\Big].
\end{eqnarray} Moreover, since the tensor $D$ is skew-symmetric in the two first indices, it is not difficult to see that $D_{kij}R^{ki}\nabla^j u =0$ and then comparing with (\ref{tensordwithu}) we infer

\begin{eqnarray} \label{normD}
D_{kij}(R^{ki}\nabla^j u  +R^{kj}\nabla^i u) & = & \frac{1}{2}D_{kij}R^{kj}\nabla^i u - \frac{1}{2}D_{ikj}R^{kj}\nabla^i u \nonumber \\
& = & -\frac{1}{2}D_{kij}(R^{ij}\nabla^k u - R^{kj}\nabla^i u) \nonumber \\
& = & \frac{n-2}{2m}u|D|^2.
\end{eqnarray}

Next, upon integrating (\ref{bachwithuandd}) over the ball $B_p(s),$ we use (\ref{normD}) together with the divergence theorem to deduce

\begin{eqnarray}\label{integralbach}
\int_{B_p(s)}B(\nabla u,\nabla u)e^{-u}u^3dV_g &=& \frac{m+n-2}{m(n-2)}\Big[ \int_{\partial B_p(s)} D_{kij}\nabla^i u \nabla^j u e^{-u}u^3\nu_kd\sigma\nonumber\\&& - \frac{n-2}{2m^2}\int_{B_p(s)} u^5|D|^2e^{-u}dV_g \Big],
\end{eqnarray} where $\nu$ denotes the outward unit normal to $\partial B_p(s).$ Moreover, since $g$ has positive Ricci curvature, then $|R_{ij}|\leq R.$ This jointly with (\ref{tensordwithu}) yields

\begin{eqnarray} \label{integralinequalityD}
\left| \int_{\partial B_p(s)} uD_{kij}\nabla^i u \nabla^j u e^{-u}u^2\nu_kd\sigma \right| & \leq & C\int_{\partial B_p(s)} |\nabla u|^3(|R_{ij}| + R) u^2 e^{-u}d\sigma \nonumber \\
& \leq & 2C\left(\sqrt{\frac{\mu}{m-1}}\right)^3 \int_{\partial B_p(s) }u^2Re^{-u}d\sigma \nonumber \\
& \leq & 2C\left(\sqrt{\frac{\mu}{m-1}}\right)^3m\mu \int_{\partial B_p(s) }e^{-u}d\sigma,
\end{eqnarray} where we also have used (\ref{nablau}) and (\ref{scalarlimitationwithu}). Moreover, we already know from (\ref{estimativateorema}) that $$-u(x) \leq -c_1r(x) + c_2,$$ where $c_{1}$ and $c_{2}$ are positive constants and $r$ is the distance function. Thus, by (\ref{integralinequalityD}) one has

\begin{equation}
\label{n1}
\left| \int_{\partial B_p(s)} uD_{kij}\nabla^i u \nabla^j u e^{-u}u^2\nu_kd\sigma \right| \leq C_{1}e^{-s}Area(\partial B_p(s)).
\end{equation} The assumption of positive Ricci curvature allows to use the Bishop-Gromov theorem to infer
\begin{equation*}
Area(\partial B_p(s)) \leq C_{2}s^{n-1}.
\end{equation*} Hence, it follows from (\ref{n1}) that

\begin{equation*}
\left| \int_{\partial B_p(s)} uD_{kij}\nabla^i u \nabla^j u e^{-u}u^2\nu_kd\sigma \right| \leq C_{3}e^{-s}s^{n-1}.
\end{equation*} Therefore, by letting $s \rightarrow +\infty$ in Eq. (\ref{integralbach}) we achieve

\begin{equation*}
\int_{M}B(\nabla u,\nabla u)e^{-u}u^3dV_g = -\frac{m+n-2}{2m^3}\int_{M} u^5|D|^2e^{-u}dV_g.
\end{equation*} Finally, since $M^n$ is Bach-flat and $u>0$ we conclude $D_{ijk}=0.$ This finishes the proof of the lemma.

\end{proof}

\section{Proof of the Main Results}

\subsection{Proof of Theorem \ref{theoremcottonvanishes}}

\begin{proof} We follow the trend of Chen and He \cite{QiangHe} (see also Cao and Chen \cite{caoshinking}). To begin with, since $M^n$ is Bach-flat it follows from Lemma \ref{bachimpliesd} that $D_{ijk} = 0.$ Therefore, we may use (\ref{cottonwithd}) to get
 \begin{equation}\label{cwdzero}
C_{ijk} = -W_{ijkl}\nabla^lf.
\end{equation}
At the same time, we already know that such a metric is real analytic (cf. Proposition  2.4 in \cite{Petersen}). Therefore, taking into account (\ref{cwdzero}) as well as (\ref{relationcottonandweyl}), it suffices to show that the Cotton tensor $C_{ijk}$ vanishes at points $p\in M^n$ such that $\nabla f (p) \neq 0.$ So, we consider a regular point $p\in M^n,$  with associated level set  $\Sigma.$ Moreover, choose any local coordinates $(\theta^{2},\ldots,\theta^{n})$ on $\Sigma$ and split the metric in the local coordinates $(f,\theta^{2},\ldots,\theta^{n})$ as follows $$g=\frac{1}{|\nabla f|^{2}}df^2+g_{ab}(f,\,\theta) d\theta^{a}d\theta^{b}.$$  Letting $\partial_{f}=\partial_{1}=\frac{\nabla f}{|\nabla f|^{2}}$ we immediately get $\nabla_{1}f=1$ and  $\nabla_{a}f=0,$ for $a\geq 2.$ From (\ref{cwdzero}) and the symmetries of the Weyl tensor we have $C_{ij1} = 0.$ Next, by Proposition \ref{levelcurvesgeometry}, we have $R_{1a} = 0$ and $R_{1abc} = 0$ for any integers $2 \leq a,b,c \leq n.$ Whence, it is easy to check that 
\begin{equation*}
W_{abc1} = R_{abc1} = 0
\end{equation*} and use once more (\ref{cwdzero}) to deduce $C_{abc} = -W_{abc1}|\nabla f |^2= 0.$

We now claim that $C_{1ab}=0$ for all $a,b\geq 2.$ To prove our claim we apply the same arguments used in \cite{QiangHe} (p. 324). Indeed, notice that
\begin{eqnarray*}
C_{1ab}&=&\frac{1}{|\nabla f|^{2}}W(\nabla f,\partial_a,\nabla f,\partial_b).
\end{eqnarray*} On the other hand, from (\ref{riemanntensor}) we infer
\begin{eqnarray}
\label{lemcvanisheq2}
\frac{1}{|\nabla f|^{2}}W(\nabla f,\partial_a,\nabla f,\partial_b)&=&\frac{1}{|\nabla f|^{2}}R(\nabla f,\partial_a,\nabla f,\partial_b)+\frac{R}{(n-1)(n-2)}g_{ab}\nonumber\\&&-\frac{1}{(n-2)}\left(\frac{1}{|\nabla f|^{2}}Ric(\nabla f,\nabla f)g_{ab}+R_{ab}\right).
\end{eqnarray} Easily one verifies that $h_{ab}=\frac{\Gamma_{ab}^{1}}{|\nabla f|}.$ Moreover, we also have $\Gamma_{ab}^{1}=-\frac{1}{2}\nabla f (g_{ab}).$ Hence, it follows that
\begin{equation}
\label{lemcvanisheq3} h_{ab}=-\frac{\nabla f}{2|\nabla f|}(g_{ab}).
\end{equation} Proceeding, we invoke Proposition \ref{levelcurvesgeometry} to deduce that $|\nabla f|$ is constant on $\Sigma,$ which immediately gives $[\partial_a,\nabla f]=0,$ and then $\big\langle \frac{\nabla f}{|\nabla f|},\partial_{a}\big\rangle=0,$ which implies $\nabla_{\frac{\nabla f}{|\nabla f|}}\frac{\nabla f}{|\nabla f|}=0.$ By these settings we get 
\begin{equation}
\label{9990}
\frac{1}{|\nabla f|^2}R(\nabla f,\partial_a,\nabla f,\partial_b)=\frac{\nabla f}{(n-1)|\nabla f|}H g_{ab}-\frac{H^2}{(n-1)^2}g_{ab}.
\end{equation} In particular, by tracing (\ref{9990}) with respect to $a$ and $b$ we obtain $$\frac{1}{|\nabla f|^2}Ric(\nabla f,\nabla f)=\frac{\nabla f}{|\nabla f|}H-\frac{H^2}{(n-1)}.$$  This substituted into (\ref{9990}) yields 
\begin{equation}\label{sol}
R(\nabla f,\partial_a,\nabla f,\partial_b)=\frac{Ric(\nabla f,\nabla f)}{(n-1)}g_{ab}.
\end{equation} By using again Proposition \ref{levelcurvesgeometry} (3)  we may consider $\frac{1}{|\nabla f|^2}Ric(\nabla f,\nabla f)=\eta$ and $Ric(\partial_a,\partial_b)=\kappa g_{ab},$ for $a,b\geq2,$ where $\eta$ and $\kappa$ are the eigenvalues of the Ricci curvature. Therefore, substituting (\ref{sol}) into (\ref{lemcvanisheq2}) we achieve $C_{1ab}=0,$ which settles our claim. 

Finally, it is not difficult to see that $C_{ijk}=0$ whenever $\nabla f (p) \neq 0.$ Besides, we already know that $g$ is analytic, which allows us to conclude that $C_{ijk}=0$ on $M^n.$ So, the proof is completed.

\end{proof}

\subsection{Proof of Theorem \ref{theoremweylvanishes}}

\begin{proof} First of all, we invoke Theorem \ref{theoremcottonvanishes} to conclude that $C \equiv 0$ and $W_{ijkl}\nabla^l f = 0.$ Moreover, we consider a point $p \in M^4$ such that $\nabla f (p) \neq 0.$ Choosing an orthonormal frame $\{e_1, e_2, e_3, e_4\}$ with $e_1 = \frac{\nabla f}{|\nabla f|}$ at the point $p,$ we have $W_{ijk1} = 0$ for all $1 \leq i, j, k \leq 4$. From now on it suffices to follow the arguments applied in the final steps of the proof of Theorem 2 in \cite{ernani} (see also \cite{caoshinking}). In fact, these steps guarantee that $W_{ijkl} = 0$ whenever $\nabla f (p) \neq 0.$ Then, since $g$ is analytic, $M^4$ is locally conformally flat. This is what we wanted to prove.  
\end{proof}

\begin{acknowledgement}
The authors would like to thank A. Barros, G. Catino and R. Batista for fruitful conversations about this subject. Moreover, the authors want to thank the referee for his careful reading and valuable suggestions. 
\end{acknowledgement}

\end{document}